\documentclass[leqno]{article}
\usepackage{amsmath,amssymb,enumerate,theorem,cite}
\usepackage[margin=1in]{geometry}
\usepackage[dvips]{graphicx,color}
\usepackage{tikz}

\makeatletter

  \@addtoreset{equation}{section}

\newenvironment{proof}[1][\proofname]{\par\normalfont
  \topsep6pt plus6pt\trivlist\item[\hskip\labelsep\itshape
  #1\@addpunct{:}]\ignorespaces}{\qed\endtrivlist}
\newcommand{\proofname}{Proof}
\DeclareRobustCommand{\qed}{%
  \ifmmode
  \else\leavevmode\unskip\penalty9999\hbox{}\nobreak\hfill\fi
  \quad\hbox{\qedsymbol}}
\newcommand{\qedsymbol}{\openbox}
\newcommand{\openbox}{\leavevmode\hbox to.77778em{%
    \hfil\vrule\vbox to.675em{%
      \hrule width.6em\vfil\hrule}\vrule\hfil}}
\makeatother

\newcommand{\Prb}{\mathsf{P}}\newcommand{\Exp}{\mathsf{E}}

\newcommand{\Lpl}{\mathcal{L}}\newcommand{\Mgf}{\mathcal{M}}
\newcommand{\dd}{\mathrm{d}}\newcommand{\ee}{\mathrm{e}}
\newcommand{\C}{\mathbb{C}}\newcommand{\R}{\mathbb{R}}
\newcommand{\N}{\mathbb{N}}\newcommand{\Z}{\mathbb{Z}}
\newcommand{\B}{\mathcal{B}}
\newcommand{\SIR}{\mathsf{SIR}}
\newcommand{\Gam}{\mathrm{Gam}}
\newcommand{\Cell}{\mathcal{C}}
\newcommand{\bsym}[1]{\boldsymbol{#1}}
\newcommand{\ind}[1]{\boldsymbol{1}_{#1}}

\let\Bar\overline\let\Tilde\widetilde
\let\geqsl\ge\let\leqsl\le

\theorembodyfont{\itshape}
\newtheorem{proposition}{Proposition}[section]
\newtheorem{theorem}{Theorem}[section]
\newtheorem{corollary}{Corollary}[section]
\newtheorem{lemma}{Lemma}[section]
\theorembodyfont{\rmfamily}
\newtheorem{assumption}{Assumption}[section]
\newtheorem{remark}{Remark}[section]

\begin{document}\sloppy\allowbreak\allowdisplaybreaks

\title{Tail asymptotics of signal-to-interference ratio distribution
  in spatial cellular network models}
\author{Naoto Miyoshi$^{\text{a)}}$\thanks{Research supported by JSPS
    Grant-in-Aid for Scientific Research~(C) 16K00030.}
  \and Tomoyuki Shirai$^{\text{b)}}$\thanks{Research supported by JSPS
    Grant-in-Aid for Scientific Research~(B) 26287019.}
}
\date{%
  \textit{Dedicated to Tomasz Rolski on the occasion of his 70th
    birthday}\\[3ex]
  $^{\text{a)}}$Department of Mathematical and Computing Science\\
  Tokyo Institute of Technology\\[1ex]
  $^{\text{b)}}$Institute of Mathematics for Industry\\
  Kyushu University
}  

\maketitle

\begin{abstract}
We consider a spatial stochastic model of wireless cellular networks,
where the base stations~(BSs) are deployed according to a simple and
stationary point process on $\R^d$, $d\geqsl 2$.
In this model, we investigate tail asymptotics of the distribution of
signal-to-interference ratio~(SIR), which is a key quantity in
wireless communications.
In the case where the path-loss function representing signal
attenuation is unbounded at the origin, we derive the exact tail
asymptotics of the SIR distribution under an appropriate sufficient
condition.
While we show that widely-used models based on a Poisson point process
and on a determinantal point process meet the sufficient condition, we
also give a counterexample violating it.
In the case of bounded path-loss functions, we derive a
logarithmically asymptotic upper bound on the SIR tail distribution
for the Poisson-based and $\alpha$-Ginibre-based models.
A logarithmically asymptotic lower bound with the same order as the
upper bound is also obtained for the Poisson-based model.
\end{abstract}

\section{Introduction and model description}\label{sec:intro}

In this paper, we consider a spatial stochastic model of downlink
cellular networks described as follows.
Let $\Phi=\{X_i\}_{i\in\N}$ denote a point process on $\R^d$, $d\geqsl 2$
(mostly $d=2$ is supposed), where the points are ordered according to
the distance from the origin such that $|X_1| \leqsl |X_2| \leqsl
\cdots$.
Each point~$X_i$, $i\in\N$, represents the location of a base
station~(BS) of the cellular network and we refer to the BS located at
$X_i$ as BS~$i$.
The point process~$\Phi$ is assumed to be simple almost surely in
probability~$\Prb$ ($\Prb$-a.s.) and stationary in $\Prb$ with
positive and finite intensity~$\lambda = \Exp\Phi([0,1]^d)$.
Assuming further that all the BSs transmit signals at the same power
level and each user is associated with the nearest BS, we focus on a
typical user located at the origin~$o=(0,0,\ldots,0)\in\R^d$.
For each $i\in\N$, let $H_i$ denote a nonnegative random variable
representing the propagation effect of fading and shadowing on the
signal from BS~$i$ to the typical user, where $H_i$, $i\in\N$, are
mutually independent and identically distributed~(i.i.d.), as well as
independent of the point process~$\Phi$.
The path-loss function representing attenuation of signals with
distance is denoted by $\ell$, which is a nonincreasing function
satisfying $\int_\epsilon^\infty r^{d-1}\,\ell(r)\,\dd r < \infty$ for
any $\epsilon>0$.
What we have in mind is, for example, $\ell(r) = r^{-d\beta}$ or
$\ell(r) = (1 + r^{d\beta})^{-1}$ with $\beta>1$, the former of which
is an example of unbounded path-loss functions and the latter is
bounded.

In this model, the signal-to-interference ratio~(SIR) for the typical
user is defined as
\begin{equation}\label{eq:SIR}
  \SIR_o = \frac{H_1\,\ell(|X_1|)}
                {\sum_{i=2}^\infty H_i\,\ell(|X_i|)},
\end{equation}
where we recall that $X_1$ is the nearest point of $\Phi$ from the
origin and the typical user at the origin is associated with BS~$1$ at
$X_1$.
We can see that $\SIR_o$ in \eqref{eq:SIR} is invariant to the
intensity~$\lambda$ of the point process~$\Phi$.
While SIR is a key quantity in design and analysis of wireless
networks, spatial cellular network models where the SIR distribution
is obtained exactly in a closed-form or a numerically computable form
are limited (see, e.g., \cite{AndrBaccGant11,MiyoShir14a}).
In addition, even when it is numerically computable, the actual
computation can be time-consuming~(\cite{MiyoShir14a}).
Several researchers therefore resort to some approximation and/or
asymptotic approaches recently (see, e.g.,
\cite{BlasKarrKeel15,GantHaen16,GuoHaen15,Haen14,KeelRossXiaBlas16,MiyoShir14b,MiyoShir16,NagaMiyoShir14}).
In the current paper, we investigate tail asymptotics of the SIR
distribution; that is, the asymptotic behavior of
$\Prb(\SIR_o>\theta)$ as $\theta\to\infty$, for the two cases where
the path-loss function is unbounded at the origin and where it is
bounded.
The part of the unbounded path-loss function is a slight refinement of
\cite{MiyoShir16} and we derive the exact tail asymptotics of the SIR
distribution under an appropriate sufficient condition.
While we show that the widely-used models on $\R^2$, where the BS
configuration~$\Phi$ is given as a homogeneous Poisson point process
and where $\Phi$ is a stationary and isotropic determinantal point
process, meet the sufficient condition, we also give a counterexample
violating it.
For the case of bounded path-loss functions, we derive a
logarithmically asymptotic upper bound on the SIR tail distribution
for the homogeneous Poisson-based and $\alpha$-Ginibre based-models,
where $\alpha$-Ginibre point processes are one of the main examples of
stationary and isotropic determinantal point processes on
$\C\simeq\R^2$.
We also derive a logarithmically asymptotic lower bound with the same
order as the upper bound for the homogeneous Poisson-based model.

The paper is organized as follows.
In the next section, we consider the BS configuration~$\Phi$ as a
general simple and stationary point process on $\R^d$ and the
path-loss function~$\ell$ as $\ell(r)=r^{-d\beta}$, $r>0$.
In this case, we derive $\Prb(\SIR_o>\theta)\sim c\,\theta^{-1/\beta}$
as $\theta\to\infty$ for some constant~$c\in(0,\infty)$ under an
appropriate sufficient condition.
In Section~\ref{sec:examples}, we show that the widely-used
Poisson-based and determinantal-based models on $\R^2$ meet the
sufficient condition while we also give a counterexample to it.
In Section~\ref{sec:bounded}, we consider bounded and regularly
varying path-loss functions and derive a logarithmically asymptotic
upper bound on $\Prb(\SIR_o>\theta)$ as $\theta\to\infty$ when the
propagation effect distribution is light-tailed and $\Phi$ is a
homogeneous Poisson point process or an $\alpha$-Ginibre point
process.
When $\Phi$ is a homogeneous Poisson point process and the propagation
effects are exponentially distributed, a logarithmically asymptotic
lower bound with the same order as the upper bound is also obtained.

\section{Tail asymptotics for unbounded path-loss
  models}\label{sec:asymptotic_result}

In this section, we consider the path-loss function $\ell(r) =
r^{-d\beta}$, $r>0$, and derive $\Prb(\SIR_o>\theta) \sim
c\,\theta^{-1/\beta}$ as $\theta\to\infty$ with some constant $c>0$
under an appropriate set of conditions.
Prior to providing the main theorem, we need a short preliminary.

Let $\Prb^o$ and $\Exp^o$ denote respectively the Palm probability and
the corresponding expectation with respect to the marked point
process~$\Phi_H = \{(X_i, H_i)\}_{i\in\N}$ viewed at the origin (see,
e.g., \cite[Sec.~1.4]{BaccBlas09a} or \cite[Chap.~13]{DaleVere08}).
Note that, due to the independence of $\Phi=\{X_i\}_{i\in\N}$ and
$\{H_i\}_{i\in\N}$, we have $\Prb^o(H_1\in C) = \Prb(H_1\in C)$ for
any $C\in\B(\R_+)$.
When we consider the point process~$\Phi$ under the Palm
distribution~$\Prb^o$, we use index~$0$ for the point at the origin;
that is, $X_0=o=(0,0,\ldots,0)\in\R^d$ under $\Prb^o$.
For the point process~$\Phi$ and a point~$X_i$ of $\Phi$, the Voronoi
cell of $X_i$ with respect to $\Phi$ is defined as
\[
  \Cell(X_i)
  = \{x\in\R^d: |x-X_i| \leqsl |x-X_j|, X_j\in\Phi\};
\]
that is, the set of points in $\R^d$ whose distance to $X_i$ is not
greater than that to any other points of $\Phi$.
The typical Voronoi cell is then $\Cell(o)$ under the Palm
distribution~$\Prb^o$ and its circumscribed radius, denoted by $R(o)$,
is the radius of the smallest ball centered at the origin and
containing $\Cell(o)$ under $\Prb^o$.

\begin{theorem}\label{thm:general}
For the cellular network model described in the preceding section with
the path-loss function~$\ell(r) = r^{-d\beta}$, $r>0$, we suppose the
following.
\begin{enumerate}[(A)]
\item\label{condA}
  For the propagation effects~$H_i$, $i\in\N$,
  $\Exp({H_1}^{1/\beta})<\infty$ and there exist $p>0$ and $c_H>0$
  such that the Laplace transform $\Lpl_H$ of $H_i$, $i\in\N$,
  satisfies $\Lpl_H(s)\leqsl c_H\,s^{-p}$ for $s\geqsl 1$.
\item\label{condB}
  For the point process~$\Phi=\{X_i\}_{i\in\N}$,
  $\Exp^o(R(o)^d)<\infty$ and there exists a $k >
  (p\,\beta)^{-1}$ such that $\Exp^o(|X_k|^d)<\infty$, where $p$ is
  that in Condition~(A) above. 
\end{enumerate}
We then have
\begin{align}\label{eq:GH15Thm4}
 \lim_{\theta\to\infty}\theta^{1/\beta}\,
   \Prb(\SIR_o > \theta)
 &= \pi_d\,\lambda\,
    \Exp({H_1}^{1/\beta})\,
    \Exp^o\biggl[
      \biggl(
        \sum_{i=1}^\infty \frac{H_i}{|X_i|^{d\,\beta}}
      \biggr)^{-1/\beta}
    \biggr],
\end{align}
where $\pi_d = \pi^{d/2}/\Gamma(d/2+1)$ denotes the volume of a
$d$-dimensional unit ball with the Gamma function $\Gamma(x) =
\int_0^\infty t^{x-1}\,\ee^{-t}\,\dd t$.
\end{theorem}

Theorem~\ref{thm:general} is a slight extention of \cite{MiyoShir16}
to higher dimensions.
Recall that $\SIR_o$ in \eqref{eq:SIR} is invariant to the
intensity~$\lambda$ of the point process~$\Phi$.
Thus, we can show that the right-hand side of \eqref{eq:GH15Thm4} does
not depend on $\lambda$ (see, e.g., the remark of Definition~4 in
\cite{GantHaen16}).

\begin{remark}
When $d=2$, the right-hand side of \eqref{eq:GH15Thm4} in
Theorem~\ref{thm:general} coincides with $\mathsf{EFIR}^\delta$ in
Theorem~4 of \cite{GantHaen16}; that is, that theorem and our
Theorem~\ref{thm:general} assert the same result.
A difference between the two theorems (besides our extention to
higher dimensions) is that we offer the set of
conditions~(A) and~(B), the role of which is discussed
in the proof and the remarks thereafter.
\end{remark}

\begin{proof}
Let $F_H$ denote the distribution function of $H_i$, $i\in\N$, and let
$\Bar{F_H}(x) = 1-F_H(x)$.
By~\eqref{eq:SIR} with $\ell(r)=r^{-d\beta}$, $r>0$, the tail
probability of the SIR for the typical user is expressed as
\begin{equation}\label{eq:BarF_H}
  \Prb(\SIR_o > \theta)
  = \Exp\Bar{F_H}\biggl(
      \theta\,|X_1|^{d\,\beta}
      \sum_{i=2}^\infty
        \frac{H_i}{|X_i|^{d\,\beta}}
    \biggr).
\end{equation}
Applying the Palm inversion formula~(see, e.g.,
\cite[Sec.~4.2]{BaccBlas09a}) to the right-hand side above, we have
\begin{align*}
  \Prb(\SIR_o > \theta)
  &= \lambda\int_{\R^d}
       \Exp^o\biggl[
         \Bar{F_H}\biggl(
           \theta\, |x|^{d\,\beta}
           \sum_{i=1}^\infty
             \frac{H_i}{|X_i-x|^{d\,\beta}}
         \biggr)\,
         \ind{\Cell(o)}(x)
       \biggr]\,
     \dd x
  \\
  &= \theta^{-1/\beta}\,\lambda
     \int_{\R^d}
       \Exp^o\biggl[
         \Bar{F_H}\biggl(
           |y|^{d\,\beta}
           \sum_{i=1}^\infty
             \frac{H_i}
                  {|X_i-\theta^{-1/(d\,\beta)}\,y|^{d\,\beta}}
         \biggr)\,
         \ind{\Cell(o)}(\theta^{-1/(d\,\beta)}\,y)
       \biggr]\,
     \dd y,
\end{align*}
where the second equality follows from the substitution of
$y=\theta^{1/(d\,\beta)}\,x$.
Therefore, if we can find a random function $A$ satisfying
\begin{gather}
  \Bar{F_H}\biggl(
    |y|^{d\,\beta}
    \sum_{i=1}^\infty
      \frac{H_i}
           {|X_i-\theta^{-1/(d\,\beta)}\,y|^{d\,\beta}}
  \biggr)\,
  \ind{\Cell(o)}(\theta^{-1/(d\,\beta)}\,y)
  \leqsl A(y),
  \quad\text{$\Prb^o$-a.s.,}
  \label{eq:A1}\\
  \int_{\R^d} \Exp^o A(y)\,\dd y < \infty,
  \label{eq:A2}
\end{gather}
the dominated convergence theorem yields
\begin{align}\label{eq:limit}
 \lim_{\theta\to\infty}
    \theta^{1/\beta}\,\Prb(\SIR_o>\theta)
 &= \lambda\int_{\R^d}
      \Exp^o\Bar{F_H}\biggl(
        |y|^{d\,\beta}
        \sum_{i=1}^\infty\frac{H_i}{|X_i|^{d\,\beta}}
      \biggr)\,
    \dd y.
\end{align}
We postpone finding such a function~$A$ and admit \eqref{eq:limit} for
a moment.
Then, substituting $z = \bigl(\sum_{i=1}^\infty
H_i/|X_i|^{d\,\beta}\bigr)^{1/(d\,\beta)}\,y$ to the integral in
\eqref{eq:limit}, we have
\begin{equation}\label{eq:integral1}
 \int_{\R^d}
   \Exp^o\Bar{F_H}\biggl(
     |y|^{d\,\beta}
     \sum_{i=1}^\infty \frac{H_i}{|X_i|^{d\,\beta}}
   \biggr)\,
 \dd y
 = \Exp^o\biggl[
     \biggl(
       \sum_{i=1}^\infty \frac{H_i}{|X_i|^{d\,\beta}}
     \biggr)^{-1/\beta}
   \biggr]
   \int_{\R^d}\Bar{F_H}(|z|^{d\,\beta})\,\dd z,
\end{equation}
and the integral on the right-hand side above further reduces to
\begin{align}\label{eq:integral2}
  \int_{\R^d}\Bar{F_H}(|z|^{d\,\beta})\,\dd z
  = d\,\pi_d\int_0^\infty\Bar{F_H}(r^{d\,\beta})\,r^{d-1}\,\dd r
  = \frac{\pi_d}{\beta}\,
    \int_0^\infty \Bar{F_H}(s)\,s^{-1+1/\beta}\,\dd s
  = \pi_d\,\Exp({H_1}^{1/\beta}).
\end{align}
Hence, applying \eqref{eq:integral1} and \eqref{eq:integral2} to
\eqref{eq:limit}, we obtain \eqref{eq:GH15Thm4}.

It remains to find a function~$A$ satisfying \eqref{eq:A1} and
\eqref{eq:A2}.
Since $\Bar{F_H}$ is nonincreasing and $|X_i-y|\leqsl |X_i|+R(o)$
$\Prb^o$-a.s.\ for $y\in\Cell(o)$, we can set a function~$A$ satisfying
\eqref{eq:A1} as
\[
  A(y) =
  \Bar{F_H}\biggl(
    |y|^{d\,\beta}\,
    \sum_{i=1}^\infty
      \frac{H_i}{(|X_i| + R(o))^{d\,\beta}}
  \biggr).
\]
We now confirm that this function~$A$ satisfies~\eqref{eq:A2}.
Substituting $z = \bigl( \sum_{i=1}^\infty H_i $ $/ (|X_i| +
R(o))^{d\,\beta} \bigr)^{1/(d\,\beta)}\,y$ and using
\eqref{eq:integral2} again, we have
\[
 \int_{\R^d}\Exp^o A(y)\,\dd y
 = \pi_d\,\Exp({H_1}^{1/\beta})\,
   \Exp^o\biggl[
     \biggl(
       \sum_{i=1}^\infty
         \frac{H_i}{(|X_i| + R(o))^{d\,\beta}}
     \biggr)^{-1/\beta}
   \biggr],
\]
where $\Exp({H_1}^{1/\beta})<\infty$ from Condition~(A).
Applying the identity $x^{-1/\beta} = \Gamma(1/\beta)^{-1}
\int_0^\infty \ee^{-x\,s}\, s^{-1+1/\beta}\,\dd s$ to the second
expectation on the right-hand side above, we have
\begin{align*}
 \Exp^o\biggl[
   \biggl(
     \sum_{i=1}^\infty
       \frac{H_i}{(|X_i| + R(o))^{d\,\beta}}
   \biggr)^{-1/\beta}
 \biggr]
 = \frac{1}{\Gamma(1/\beta)}
   \int_0^\infty
     s^{-1+1/\beta}\,
     \Exp^o\biggl[
       \prod_{i=1}^\infty
         \Lpl_H\biggl(
           \frac{s}{(|X_i| + R(o))^{d\,\beta}}
         \biggr)
     \biggr]\,
   \dd s.
\end{align*} 
Recall that $X_i$, $i\in\N$, are ordered such that
$|X_1|<|X_2|<\cdots$.
By truncating the infinite product above by a finite~$k\in\N$ such
that $p\,\beta\,k>1$ and applying $\Lpl_H(s)\leqsl c_H\,s^{-p}$ for
$s\geqsl 1$ from Condition~(A), we can bound the integral on the
right-hand side above by
\begin{align*}
  &\int_0^\infty
     s^{-1+1/\beta}\,
     \Exp^o\biggl[
       \prod_{i=1}^k
         \Lpl_H\biggl(
           \frac{s}{(|X_i| + R(o))^{d\,\beta}}
         \biggr)
     \biggr]\,
   \dd s
  \\
  &\leqsl \int_0^\infty
         s^{-1+1/\beta}\,
         \Exp^o\biggl[
           \biggl\{\Lpl_H\biggl(
             \frac{s}{(|X_k| + R(o))^{d\,\beta}}
           \biggr)\biggr\}^k
         \biggr]\,
       \dd s
  \\
  &\leqsl \Exp^o\biggl[
         \int_0^{(|X_k| + R(o))^{d\,\beta}}
           s^{-1+1/\beta}\,
         \dd s
       \biggr]
       + {c_H}^k\,
         \Exp^o\biggl[
           (|X_k| + R(o))^{d\,p\,\beta\,k}
           \int_{(|X_k| + R(o))^{d\,\beta}}^\infty
             s^{-1 + 1/\beta - p\,k}\,
           \dd s
         \biggr]
  \\
  &= \beta\,
     \biggl(
       1 + \frac{{c_H}^k}{p\,\beta\,k-1}
     \biggr)\,
     \Exp^o\bigl[
       (|X_k| + R(o))^d
     \bigr].
\end{align*}
Hence, inequality $(a+b)^d \leqsl 2^{d-1}\,(a^d + b^d)$ ensures
\eqref{eq:A2} under Condition~(B) of the theorem.
\end{proof}

\begin{remark}
The differences between the proof in \cite{GantHaen16} and ours are as
follows.
The first and less essential one is that, in~\cite{GantHaen16}, they
arrange the right-hand side of \eqref{eq:BarF_H} into a certain
appropriate form and then apply the Campbell-Mecke formula~(see, e.g.,
\cite[Sec.~1.4]{BaccBlas09a}).
On the other hand, we apply the Palm inversion formula directly.
Second, \cite{GantHaen16} does not specify any condition under which
the result holds.
However, equality~\eqref{eq:limit} requires some kind of uniform
integrability condition to change the order of the limit and
integrals.
Our set of conditions~(A) and~(B) gives a sufficient condition for
this order change to be valid and complements the proof of
\cite{GantHaen16}.
\end{remark}

\begin{remark}
Condition~(A) claims that the Laplace transform of $H_i$, $i\in\N$,
decays faster than or equal to power laws.
Though this condition excludes distributions with a mass at the
origin, it covers many practical distributions.
For example, Gamma distribution~$\Gam(p, q)$, $p>0$, $q>0$, has the
Laplace transform~$\Lpl_H(s)=(1+q\,s)^{-p}$ and we can take $c_H\geqsl
q^{-p}$.
In addition, we can see from the results of \cite{AsmuJensRoja16} that
lognormal distributions also satisfy Condition~(A).
\end{remark}

The asymptotic constant in \eqref{eq:GH15Thm4} of
Theorem~\ref{thm:general} depends on the point process~$\Phi$ and the
distribution~$F_H$ of the propagation effects.
The following proposition indicates an impact of the propagation
effect distribution on the asymptotic constant by comparing with the
case without propagation effects.

\begin{proposition}\label{thm:inequality}
Let $C(\beta, F_H)$ denote the limit on the right-hand side of
\eqref{eq:GH15Thm4}, specifying the dependence on $\beta$ and the
distribution~$F_H$ of propagation effects.
When $\Exp H_1 <\infty$, we have
\begin{equation}\label{eq:AsymIneq}
  C(\beta, F_H)
  \geqsl \frac{\Exp({H_1}^{1/\beta})}{(\Exp H_1)^{1/\beta}}\,
      C(\beta, \delta_1),
\end{equation}
where $\delta_1$ denotes the Dirac measure with the mass at $1$.
\end{proposition}

\begin{proof}
The result immediately follows from Jensen's inequality conditioned on
$\Phi=\{X_i\}_{i\in\N}$.
On the right-hand side of \eqref{eq:GH15Thm4}, since
$f(x)=x^{-1/\beta}$ is convex for $x>0$,
\begin{align*}
  \Exp^o\biggl[
    \biggl(
      \sum_{i\in\N}\frac{H_i}{|X_i|^{d\,\beta}}
    \biggr)^{-1/\beta}
  \biggr]
  \geqsl \Exp^o\biggl[
        \biggl(
          \sum_{i\in\N}\frac{\Exp H_i}{|X_i|^{d\,\beta}}
        \biggr)^{-1/\beta}
      \biggr]
  = \frac{1}{(\Exp H_1)^{1/\beta}}\,
    \Exp^o\biggl[
      \biggl(
        \sum_{i\in\N} \frac{1}{|X_i|^{d\,\beta}}
      \biggr)^{-1/\beta}
    \biggr],
\end{align*}
and \eqref{eq:AsymIneq} holds.
\end{proof}

\begin{remark}
When $F_H=\mathrm{Exp}(1)$, denoting the exponential distribution with
unit mean (which assumes Rayleigh fading and ignores shadowing), the
result of Proposition~\ref{thm:inequality} reduces to the second part
of Theorem~2 in~\cite{MiyoShir15} since $\Exp({H_1}^{1/\beta}) =
\Gamma(1+1/\beta)$ in this case (though only the Ginibre point
process~$\Phi$ is considered there).
In inequality~\eqref{eq:AsymIneq}, it is easy to see (by Jensen's
inequality for concave function $f(x)=x^{1/\beta}$) that the
coefficient~$\Exp({H_1}^{1/\beta})/(\Exp H_1)^{1/\beta}$ is not
greater than $1$.
Now, suppose that $\Exp H_1 = 1$.
Then, the dominated convergence theorem (due to ${H_1}^{1/\beta}\leqsl 1
+ H_1$ a.s.) leads to $\Exp({H_1}^{1/\beta})\to1$ as both
$\beta\downarrow1$ and $\beta\uparrow\infty$, which implies that
$C(\beta, F_H)$ tends to be larger than or equal to $C(\beta, \delta_1)$
when $\beta$ is close to $1$ or sufficiently large.
\end{remark}

\section{Examples for unbounded path-loss models}\label{sec:examples}

In this section, we restrict ourselves to the case of $d=2$ and provide
a few examples demonstrating Theorem~\ref{thm:general} in the
preceding section.
We also give a counterexample violating Condition~(B) of the theorem.

\subsection{Poisson-based model}

We here consider the BS configuration~$\Phi$ as a homogeneous Poisson
point process on $\R^2$ with positive and finite intensity.
We first confirm that $\Phi$ satisfies Condition~(B) of
Theorem~\ref{thm:general}.

\begin{lemma}\label{lem:Poi_A}
Let $\Phi=\{X_i\}_{i\in\N}$ denote a homogeneous Poisson point process
on $\R^2$ with positive and finite intensity.
Then, for $\epsilon>0$,
\begin{align}
  \Exp^o\ee^{\epsilon|X_k|} &<\infty,
  \quad k\in\N,
  \label{eq:Poi_distance}\\
  \Exp^o\ee^{\epsilon R(o)} &<\infty.
  \label{eq:Poi_radius}
\end{align}
\end{lemma}

This lemma ensures that $|X_k|$, $k\in\N$, and $R(o)$ have any order
of moments.

\begin{proof}
Let $\lambda$ denote the intensity of $\Phi$ and let $D_r$ denote the
disk centered at the origin with radius~$r>0$.
Recalling that $X_i$, $i\in\N$, are ordered such that
$|X_1|<|X_2|<\cdots$, we have
\begin{align*}
  \Prb^o(|X_k| > r)
  &= \Prb(|X_k| > r)
   = \Prb(\Phi(D_r) < k)
   = \ee^{-\lambda\pi r^2}
     \sum_{j=0}^{k-1} \frac{(\lambda\,\pi\,r^2)^j}{j!}.
\end{align*}
Thus, we can use the density function of $|X_k|$ and show
\eqref{eq:Poi_distance}.

On the other hand, for the circumscribed radius~$R(o)$ of the typical
Voronoi cell of $\Phi$, Calka~\cite[Theorem~3]{Calk02} shows that
there exists an $r_0\in(0,\infty)$ such that
\[
  \Prb^o(R(o) > r)
  \leqsl 4\,\pi\,\lambda\,r^2\,\ee^{-\pi\lambda r^2}
  \quad\text{for $r\geqsl r_0$,}
\]
and we can show \eqref{eq:Poi_radius} by applying $\Exp^o\ee^{\epsilon
  R(o)} = 1 + \epsilon \int_0^\infty \ee^{\epsilon r}\,\Prb^o(R(o) >
r)\, \dd r$.
\end{proof}

Now, we apply Theorem~\ref{thm:general} and obtain the following.

\begin{corollary}\label{cor:Poisson}
Suppose that $\Phi=\{X_i\}_{i\in\N}$ is a homogeneous Poisson point
process on $\R^2$.
When the propagation effects~$H_i$, $i\in\N$, satisfy
Condition~(A) of Theorem~\ref{thm:general}, the right-hand
side of \eqref{eq:GH15Thm4} reduces to $(\beta/\pi)\,\sin(\pi/\beta)$.
\end{corollary}

\begin{proof}
Since the conditions of Theorem~\ref{thm:general} are fulfilled, the
result follows from the proof of Lemma~6 in \cite{GantHaen16}.
\end{proof}

\begin{remark}
The asymptotic result in Corollary~\ref{cor:Poisson} agrees with that
in Remark~4 of \cite{MiyoShir14a}, where only Rayleigh fading is
considered.
Corollary~\ref{cor:Poisson} states that the SIR tail probability in
the homogeneous Poisson-based model is asymptotically insensitive to
the distribution of propagation effects as far as it satisfies
Condition~(A) of Theorem~\ref{thm:general}.
\end{remark}

\subsection{Determinantal-based model}

In this subsection, we consider $\Phi$ as a general stationary and
isotropic determinantal point process on $\C\simeq\R^2$ with
intensity~$\lambda$.
Let $K$:~$\C^2\to\C$ denote the kernel of $\Phi$ with
respect to the Lebesgue measure.
The product density functions~(joint intensities)~$\rho_n$, $n\in\N$,
with respect to the Lebesgue measure are given by
\[
  \rho_n(z_1,z_2,\ldots,z_n)
  = \det\bigl(K(z_i, z_j)\bigr)_{i, j=1,2,\ldots,n}
  \quad\text{for $z_1$, $z_2$, \ldots, $z_n\in\C$,}
\]
where $\det$ denotes the determinant.  
In order for the point process~$\Phi$ to be well-defined, we assume
that (i)~the kernel~$K$ is continuous on $\C\times\C$, (ii)~$K$ is
Hermitian in the sense that $K(z,w) = \Bar{K(w,z)}$ for $z$, $w\in\C$,
where $\Bar{z}$ denotes the complex conjugate of $z\in\C$, and
(iii)~the integral operator on $L^2(\C)$ corresponding to $K$ is of
locally trace class with the spectrum in $[0,1]$; that is, for a
compact set $C\in\B(\C)$, the restriction $K_C$ of $K$ on $C$ has the
eigenvalues $\kappa_{C,i}$, $i\in\N$, satisfying
$\sum_{i\in\N}\kappa_{C,i}<\infty$ and $\kappa_{C,i}\in[0,1]$ for each
$i\in\N$ (see, e.g., \cite[Chap.~4]{HougKrisPereVira09}).
Furthermore, for stationarity and isotropy, the kernel~$K$ is assumed
to satisfy $K(z,w) = K(0,z-w)$ which depends only on the
distance~$|z-w|$ of $z$ and $w\in\C$.
The product density functions~$\rho_n$, $n\in\N$, are then
motion-invariant (invariant to translations and rotations), and we
have $\rho_1(z)=K(z,z)=\lambda$ and that $\rho_2(0,z) = \lambda^2 -
|K(0,z)|^2$ depends only on $|z|$  for $z\in\C$.
An $\alpha$-Ginibre point process with $\alpha\in(0,1]$ is one of the
main examples of stationary and isotropic determinantal point
processes on $\C$ and its kernel is given by
\begin{equation}\label{eq:alpha_kernel}
  K_\alpha(z, w) =
  \frac{1}{\pi}\,\ee^{-(|z|^2+|w|^2)/(2\alpha)}\,\ee^{z\,\Bar{w}/\alpha},
  \quad z, w\in\C,\; \alpha\in(0,1],
\end{equation}
with respect to the Lebesgue measure (see, e.g.,
\cite{Gold10,MiyoShir16b}).
We can see that the intensity and the second product density of the
$\alpha$-Ginibre point process are $\lambda = \rho_1^{(\alpha)}(0) =
\pi^{-1}$ and $\rho_2^{(\alpha)}(0,z) =
(1-\ee^{-|z|^2/\alpha})/\pi^2$, respectively.

First, concerning Condition~(B) of
Theorem~\ref{thm:general}, we show the following.

\begin{lemma}\label{lem:determinantal}
Let $\Phi$ denote a stationary and isotropic determinantal point
process on $\C$ with positive and finite intensity as described above.
\begin{enumerate}[(i)]
\item Let $X_i$, $i\in\N$, denote the points of $\Phi$ such that
  $|X_1|<|X_2|<\cdots$.
  Then, there exist $a_1>0$ and $a_2>0$ such that, for any $k\in\N$,
  we can take an $r_k>0$ satisfying
  \begin{equation}\label{eq:bound1}
    \Prb^o(|X_k| > r) \leqsl a_1\,\ee^{-a_2\,r^2}
    \quad\text{for $r \geqsl r_k$.}
  \end{equation}
\item Let $R(o)$ denote the circumscribed radius of the typical Voronoi
  cell~$\Cell(o)$ of $\Phi$.
  Then, there exist $b_1>0$ and $b_2>0$ such that
  \begin{equation}\label{eq:bound2}
    \Prb^o(R(o) > r) \leqsl b_1\,\ee^{-b_2\,r^2}
    \quad\text{for $r>0$.}
  \end{equation}
\end{enumerate}
\end{lemma} 

By Lemma~\ref{lem:determinantal}, it is easy to confirm, similar to
Lemma~\ref{lem:Poi_A}, that $|X_k|$, $k\in\N$, and $R(o)$ have any
order of moments under $\Prb^o$.
To prove Lemma~\ref{lem:determinantal}, we use the following
supplementary lemma.

\begin{lemma}\label{lem:kernel}
The kernel $K$ of a determinantal point process~$\Phi$ satisfies
\begin{equation}\label{eq:boundM}
  \int_{\C} |K(0,z)|^2\,\dd z
  \leqsl K(0,0).
\end{equation}
\end{lemma}

\begin{proof}
For a compact set $C\in\B(\C)$ such that $0\in C$, let $K_C$ denote
the restriction of $K$ on $C$.
Let also $\kappa_{C,i}$ and $\varphi_{C,i}$, $i\in\N$, denote
respectively the nonzero eigenvalues of $K_C$ and the corresponding
orthonormal eigenfunctions; that is,
\begin{equation}\label{eq:orthonormal}
  \int_C
    \varphi_{C,i}(z)\,\Bar{\varphi_{C,j}(z)}\,
  \dd z
  = \begin{cases}
      1 &\quad\text{for $i=j$,}\\
      0 &\quad\text{for $i\ne j$.}
    \end{cases}%
\end{equation}
Then Mercer's theorem states that the following spectral expansion
holds (see, e.g., \cite{Merc09});
\begin{equation}\label{eq:Mercer}
  K_C(z,w)
  = \sum_{i=1}^\infty
      \kappa_{C,i}\,\varphi_{C,i}(z)\,\Bar{\varphi_{C,i}(w)},
  \quad z, w\in C.
\end{equation}
Thus we have
\begin{align*}
  \int_C|K(0,z)|^2\,\dd z
  &= \int_C|K_C(0,z)|^2\,\dd z
   = \sum_{i=1}^\infty
       {\kappa_{C,i}}^2\,|\varphi_{C,i}(0)|^2
  \leqsl K_C(0,0) = K(0,0),
\end{align*}
where the second equality follows from \eqref{eq:orthonormal} and
\eqref{eq:Mercer}, the inequality holds since $\kappa_{C,i}\in(0,1]$,
$i\in\N$, and the last equality follows since $0\in C$.
Finally, letting $C\uparrow\C$, we obtain \eqref{eq:boundM}.
\end{proof}

Note that Lemma~\ref{lem:kernel} implies that $\int_{\C}
|K(0,z)|^2\,\dd z \leqsl \lambda$ in our stationary case with
intensity~$\lambda\in(0,\infty)$.
Using this, we prove Lemma~\ref{lem:determinantal} as follows.

\begin{proof}[Proof of Lemma~\ref{lem:determinantal}]
Let $\Prb^!$ denote the reduced Palm probability with respect to the
marked point process $\Phi_H=\{(X_i,H_i)\}_{i\in\N}$ and let $C$
denote a bounded set in $\B(\C)$.
Since a determinantal point process is also determinantal under the
(reduced) Palm distribution (see, e.g., \cite{ShirTaka03}), $\Phi(C)$
under $\Prb^!$ has the same distribution as $\sum_{i\in\N} B_{C,i}$
with certain mutually independent Bernoulli random variables
$B_{C,i}$, $i\in\N$ (see, e.g., \cite[Sec.~4.5]{HougKrisPereVira09}).
Thus, the Chernoff-Hoeffding bound for an infinite sum with finite
mean~(see, e.g., \cite{Cher52,Hoef63} for a finite sum) states that,
for any $\epsilon\in[0,1)$, there exists a
  $c_\epsilon>0$ such that
\begin{equation}\label{eq:C-H}
  \Prb^!\bigl(
    \Phi(C) \leqsl \epsilon\,\Exp^!\Phi(C)
  \bigr)
  \leqsl \ee^{-c_\epsilon\,\Exp^!\Phi(C)},
\end{equation}
where $\Exp^!$ denotes the expectation with respect to $\Prb^!$.
On the other hand, the kernel of $\Phi$ under the reduced Palm
distribution is given by (see \cite{ShirTaka03})
\[
  K^!(z,w)
  = \frac{K(z,w)\,K(0,0) - K(z,0)\,K(0,w)}{K(0,0)},
  \quad z,w\in\C,
\]
whenever $K(0,0)>0$, which is ensured in our stationary case with
$K(0,0)=\lambda$.
Therefore, the intensity function of $\Phi$ under $\Prb^!$ reduces to
\begin{equation}\label{eq:1-corr}
  \rho^!_1(z)
  = K^!(z,z)
  = \lambda - \frac{|K(0,z)|^2}{\lambda},
\end{equation}
so that, Lemma~\ref{lem:kernel} with $K(0,0)=\lambda$ yields
\begin{equation}\label{eq:Palmmean}
  \Exp^!\Phi(C)
  = \int_C \rho^!_1(z)\,\dd z
  \geqsl \lambda\,\mu(C) - 1,
\end{equation}
where $\mu$ denotes the Lebesgue measure on $\C$.

\paragraph{\normalfont\textit{Proof of (i):}}
Note that $\Prb^o(|X_k| > r) = \Prb^!(\Phi(D_r) \leqsl k-1)$.
Since $\Exp^!\Phi(D_r) \geqsl \lambda\,\pi\,r^2 - 1$ from
\eqref{eq:Palmmean}, applying this to \eqref{eq:C-H} yields
\[
  \Prb^!\bigl(
    \Phi(D_r) \leqsl \epsilon\,( \lambda\,\pi\,r^2 - 1)
  \bigr)
  \leqsl \ee^{c_\epsilon}\,\ee^{-c_\epsilon\,\lambda\,\pi\,r^2}\!.
\]
Hence, for any $\epsilon\in(0,1)$ and $k\in\N$, we can take $r_k>0$
satisfying $\epsilon\,(\lambda\,\pi\,{r_k}^2 - 1)\geqsl k-1$, which
implies \eqref{eq:bound1}.

\paragraph{\normalfont\textit{Proof of (ii):}}
We here derive an upper bound on $\Prb^o(R(o)>r)$ by exploiting Foss
\& Zuyev's seven petals~\cite{FossZuye96}, which are considered to
obtain an upper bound on the tail distribution of the circumscribed
radius of the typical Poisson-Voronoi cell.
Consider a collection of seven disks with a common radius~$r$ centered
at points~$(r, 2\pi k/7)$, $k=0,1,\ldots,6$, in polar coordinates.
The petal~$0$ is given as the intersection of the two disks centered
at $(r,0)$, $(r, 2\pi/7)$ and the angular domain between the rays
$\phi=0$ and $\phi=2\pi/7$.
The petal~$k$ is the rotation copy of petal~$0$ by angle $2\pi k/7$
for $k=1,2,\ldots,6$ (see Figure~\ref{fig:petals}).
Let $\mathcal{P}_{r,k}$, $k=0,1,\ldots,6$, denote the set formed by
petal~$k$ on the complex plane~$\C$.
Then, according to the discussion in the proof of Lemma 1 of
\cite{FossZuye96},
\begin{align}\label{eq:FoZu96}
  \Prb^o(R(o) > r)
  &\leqsl \Prb^!\biggl(
            \bigcup_{k=0}^6\{\Phi(\mathcal{P}_{r,k})=0\}
          \biggr)
  \leqsl 7\,\Prb^!(\Phi(\mathcal{P}_{r,0})=0),
\end{align}
where the second inequality follows from the isotropy of $\Phi$ under
the Palm distribution.
Now, we can apply inequality~\eqref{eq:C-H} with $\epsilon=0$ and we
have
\begin{equation}\label{eq:C-H-petal}
  \Prb^!( \Phi(\mathcal{P}_{r,0}) = 0 )
  \leqsl \ee^{-c_0\,\Exp^!\Phi(\mathcal{P}_{r,0})}.
\end{equation}
Hence, \eqref{eq:bound2} holds since $\Exp^!\Phi(\mathcal{P}_{r,0})
\geqsl \lambda\,\mu(\mathcal{P}_{r,0}) - 1$ and $\mu(\mathcal{P}_{r,0}) =
2\,r^2\,(\pi/7 + \sin(\pi/7)\,\cos(3\,\pi/7))$.
\end{proof}

\begin{figure}
\begin{center}
\begin{tikzpicture}[scale=.6]
\draw[->](-6.5,0)--(6.5,0);
\draw[->](0,-6.5)--(0,6.5);
\foreach \x in {0, 51.42857142857143, 102.8571428571429,
  154.2857142857143, 205.7142857142857, 257.1428571428571,
  308.5714285714286}
  \draw[rotate=\x] (3,0) circle (3cm);
\fill[green!20!white] (0,0) circle (3.74093881115cm);
\foreach \x in {0, 51.42857142857143, 102.8571428571429,
  154.2857142857143, 205.7142857142857, 257.1428571428571,
  308.5714285714286}
  \fill[rotate=\x,rotate around={51.4285714286:(3,0)},green!20!white]
  (3cm,0cm)--(6cm,0cm) arc (0:51.4285714286:3cm);
\foreach \x in {0, 51.42857142857143, 102.8571428571429,
  154.2857142857143, 205.7142857142857, 257.1428571428571,
  308.5714285714286}
  \fill[rotate=\x,rotate around={257.1428571428571:(3,0)},green!20!white]
    (3cm,0cm)--(6cm,0cm) arc (0:51.4285714286:3cm);
\foreach \x in {0, 51.42857142857143, 102.8571428571429,
  154.2857142857143, 205.7142857142857, 257.1428571428571,
  308.5714285714286}
  \draw[rotate=\x,rotate around={51.4285714286:(3,0)}]
  (6cm,0cm) arc (0:51.4285714286:3cm);
\foreach \x in {0, 51.42857142857143, 102.8571428571429,
  154.2857142857143, 205.7142857142857, 257.1428571428571,
  308.5714285714286}
  \draw[rotate=\x,rotate around={257.1428571428571:(3,0)}]
    (6cm,0cm) arc (0:51.4285714286:3cm);
\foreach \x in {0, 51.42857142857143, 102.8571428571429,
  154.2857142857143, 205.7142857142857, 257.1428571428571,
  308.5714285714286}
  \draw[rotate=\x](0,0)--(3.74093881115,0);
\foreach \x in {0, 51.42857142857143, 102.8571428571429,
  154.2857142857143, 205.7142857142857, 257.1428571428571,
  308.5714285714286}
  \filldraw[rotate=\x] (3,0) circle (.7mm);
\draw (3,0) node[below,inner sep=1.5mm]{$r$};
\draw (7mm,0mm) arc (0:51.4285714286:7mm) node[inner xsep=1.5mm,right]{$\frac{2\pi}{7}$};
\end{tikzpicture}
\end{center}
\caption{Foss \& Zuyev's seven petals~(\cite{FossZuye96}).}\label{fig:petals}
\end{figure}
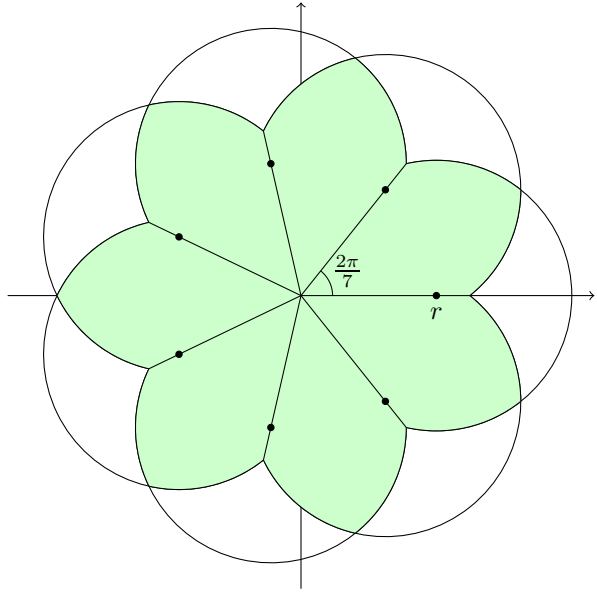

\begin{remark}
The first part~(i) of Lemma~\ref{lem:determinantal} (as well as the
first part~\eqref{eq:Poi_distance} of Lemma~\ref{lem:Poi_A}) can be
extended to a determinantal point process on $\R^d$
(see~\cite[Lemma~5.6]{BlasYogeYuki16}).
We can take $c_0$ in \eqref{eq:C-H-petal} equal to $1$ since
determinantal point processes are weakly sub-Poisson (in particular,
due to the $\nu$-weakly sub-Poisson property) (see \cite{BlasYoge14}
for details).
\end{remark}

\begin{remark}
When the kernel~$K$ of a determinantal point process is explicitly
specified, it may be possible to obtain a tighter upper bound on the
tail probability of the circumscribed radius of the typical Voronoi
cell.
For example, the case of an $\alpha$-Ginibre point process is given by
the following corollary.
\end{remark}

\begin{corollary}\label{cor:Ginibre-radius}
For an $\alpha$-Ginibre point process, the circumscribed radius for
the typical Voronoi cell $\mathcal{C}(o)$ satisfies
\begin{equation}\label{eq:covradius}
  \Prb^0(R(o) > r)
  \le 7\,\ee^{-(u_\alpha(r)\vee v_\alpha(r))},
\end{equation}
where
\begin{align*}
  u_\alpha(r)
  &= \frac{1}{7}\,
     \Bigl\{
       4r^2\,\cos^2\frac{2\pi}{7}
       - \alpha\,\Bigl[
           1 - \exp\Bigl(
                 -\frac{4 r^2}{\alpha}\,\cos^2\frac{2\pi}{7}
               \Bigr)
         \Bigr]           
     \Bigr\},
  \\
  v_\alpha(r)
  &= \frac{2 r^2}{\pi}\,
     \Bigl(
       \frac{\pi}{7} + \sin\frac{\pi}{7}\,\cos\frac{3 \pi}{7}
     \Bigr)
     - \frac{\alpha}{7}\,
       \Bigl[
         1 - \exp\Bigl(
               - \frac{4 r^2}{\alpha}\,\cos^2\frac{\pi}{7}
             \Bigr)
       \Bigr].
\end{align*}
\end{corollary}

\begin{proof}
By the kernel of the $\alpha$-Ginibre point process
in~\eqref{eq:alpha_kernel}, the intensity function
of~\eqref{eq:1-corr} under the (reduced) Palm distribution reduces to
\begin{equation}\label{eq:Ginibre-Palm-intensity}
  \rho_1^!(z)
  = \frac{1}{\pi}\,\bigl(1 -  \ee^{-|z|^2/\alpha}\bigr),
  \quad z\in\C.
\end{equation}
We obtain two lower bounds of $\Exp^!\Phi(\mathcal{P}_{r,0})$ as
follows.
Let $\mathcal{S}_\eta$ denote the circular sector centered at the
origin with radius~$\eta$ and the angular domain between $\phi=0$ and
$\phi=2\,\pi/7$.
Taking $\eta_1=2r\,\cos(2\pi/7)$ and $\eta_2=2r\,\cos(\pi/7)$, we have
$\mathcal{S}_{\eta_1}\subset\mathcal{P}_{r,0}\subset\mathcal{S}_{\eta_2}$.
Therefore, applying \eqref{eq:Ginibre-Palm-intensity}, we have the
first lower bound;
\begin{align*}
  \Exp^!\Phi(\mathcal{P}_{0,r})
 &\ge \Exp^!\Phi(\mathcal{S}_{\eta_1})
  = \int_{\mathcal{S}_{\eta_1}} \rho^!_1(z)\,\dd z
  = \frac{1}{7}\,
    \bigl[
      {\eta_1}^2 + \alpha\,(\ee^{-{\eta_1}^2/\alpha} - 1)
    \bigr]= u_\alpha(r).
\end{align*}
The second lower bound is given by
\begin{align*}
  \Exp^!\Phi(\mathcal{P}_{0,r})
 &=  \int_{\mathcal{P}_{0,r}}\rho_1^!(z)\,\dd z
  \ge \frac{1}{\pi}\,
      \Bigl(
        \mu(\mathcal{P}_{r,0}) -
        \int_{\mathcal{S}_{\eta_2}}\ee^{-|z|^2/\alpha}\,\dd z
      \Bigr) = v_\alpha(r).
\end{align*}
Hence, we have \eqref{eq:covradius} from \eqref{eq:FoZu96} and
\eqref{eq:C-H-petal} with $c_0=1$.
\end{proof}

Indeed, when $\alpha=1$ for example, we can numerically compute 
$r_*\approx0.5276\cdots$ such that $u_1(r)>v_1(r)$ for $r<r_*$ and
$u_1(r)<v_1(r)$ for $r>r_*$.
We are now ready to give the tail asymptotics of the SIR distribution
when the BSs are deployed according to an $\alpha$-Ginibre point
process.

\begin{corollary}\label{cor:Ginibre}
Suppose that $\Phi=\{X_i\}_{i\in\N}$ is an $\alpha$-Ginibre point
process.
When the propagation effects~$H_i$, $i\in\N$, satisfy Condition~(A) of
Theorem~\ref{thm:general}, we have
\begin{align}\label{eq:GiniAsym}
  \lim_{\theta\to\infty}\theta^{1/\beta}\,
    \Prb(\SIR_o > \theta)
 &= \frac{\alpha\,\Exp({H_1}^{1/\beta})}{\Gamma(1+1/\beta)}
    \int_0^\infty
      \prod_{i=1}^\infty
        \biggl[
          1 - \alpha
          + \frac{\alpha}{i!}
            \int_0^\infty
              \ee^{-y}\,y^i\,
              \Lpl_H\Bigl(\Bigl(\frac{t}{y}\Bigr)^\beta\Bigr)\,
            \dd y\,
        \biggr]
    \dd t.
\end{align}
\end{corollary}

For the proof of Corollary~\ref{cor:Ginibre}, we use the following
proposition which is a consequence of \cite{Gold10} and \cite{Kost92}
(see also \cite{MiyoShir16b}).

\begin{proposition}\label{prp:Kostlan}
\begin{enumerate}[(i)]  
\item Let $X_i$, $i\in\N$, denote the points of an $\alpha$-Ginibre
  point process.
  Then, the set $\{|X_i|^2\}_{i\in\N}$ has the same distribution as
  $\check{\bsym{Y}} = \{\check{Y_i}\}_{i\in\N}$, which is extracted
  from $\bsym{Y} = \{Y_i\}_{i\in\N}$ such that $Y_i$, $i\in\N$, are
  mutually independent with $Y_i\sim\Gam(i, \alpha^{-1})$ for each
  $i\in\N$ and each $Y_i$ is added in $\check{\bsym{Y}}$ with
  probability~$\alpha$ and discarded with $1-\alpha$ independently of
  others.
\item Let $X_i$, $i\in\N$, denote the points of an $\alpha$-Ginibre
  point process under the reduced Palm distribution.
  Then, the same statement as (i) holds except for replacing
  $Y_i\sim\Gam(i, \alpha^{-1})$ by $Y_i\sim\Gam(i+1, \alpha^{-1})$.
\end{enumerate}
\end{proposition}

\begin{proof}[Proof of Corollary~\ref{cor:Ginibre}]
For an $\alpha$-Ginibre point process, we can see by
Lemma~\ref{lem:determinantal} (or
Corollary~\ref{cor:Ginibre-radius}) that $|X_k|$, $k\in\N$, and
$R(o)$ have any order of moments under the Palm distribution~$\Prb^o$;
that is, Condition~(B) of Theorem~\ref{thm:general} is
fulfilled.
Thus, applying the identity $x^{-1/\beta} =
\Gamma(1/\beta)^{-1}\int_0^\infty \ee^{-x\,s}\, s^{-1+1/\beta}\,\dd s$
and the Laplace transform~$\Lpl_H$ to the right-hand side of
\eqref{eq:GH15Thm4}, we have
\begin{align*}
  \Exp^0\biggl[
    \biggl(
      \sum_{i=1}^\infty\frac{H_i}{|X_i|^{2\,\beta}}
    \biggr)^{-1/\beta}
  \biggr]
  &= \frac{1}{\Gamma(1/\beta)}
     \int_0^\infty
       s^{-1+1/\beta}\,
       \Exp^0\biggl[
         \prod_{i=1}^\infty
           \Lpl_H\biggl(
             \frac{s}{|X_i|^{2\,\beta}}
           \biggr)
       \biggr]\,
     \dd s
  \\
  &= \frac{1}{\Gamma(1+1/\beta)}
     \int_0^\infty
       \Exp^0\biggl[
         \prod_{i=1}^\infty
           \Lpl_H\biggl(
             \biggl(\frac{t}{|X_i|^2}\biggr)^\beta
           \biggr)
       \biggr]\,
     \dd t,
  \nonumber
\end{align*}
where the second equality follows by substituting $t=s^{1/\beta}$.
Here, applying $\bsym{Y} = \{Y_i\}_{i\in\N}$ in
Proposition~\ref{prp:Kostlan}~(ii), we have \eqref{eq:GiniAsym}.
\end{proof}

\begin{remark}
When $F_H=\mathrm{Exp}(1)$ (Rayleigh fading without shadowing),
\eqref{eq:GiniAsym} reduces to the result of Theorem~1 in
\cite{MiyoShir14b}.
When $F_H = \Gam(m,\:1/m)$ (Nakagami-$m$ fading without shadowing), we
have $\Lpl_H(s)=(1+s/m)^{-m}$ and $\Exp({H_1}^{1/\beta}) =
\Gamma(m+1/\beta)/(m^{1/\beta}\,(m-1)!)$.
Applying these to the right-hand side of \eqref{eq:GiniAsym} yields
\begin{align*}
  \lim_{\theta\to\infty}\theta^{1/\beta}\,
    \Prb(\SIR_o > \theta)
  &= \frac{\alpha\,\Gamma(m+1/\beta)}
          {\Gamma(1+1/\beta)\,m^{1/\beta}\,(m-1)!}
     \int_0^\infty
       \prod_{i=1}^\infty
         \biggl[
           1 - \alpha
             + \frac{\alpha}{i!}
               \int_0^\infty
                 \frac{\ee^{-y}\,y^i}
                      {\bigl(1 + m^{-1}\,(t/y)^\beta\bigr)^m}\,
               \dd y
         \biggr]\,                      
     \dd t
  \\
  &= \frac{\alpha\,\beta}{B(m,\:1/\beta)}
    \int_0^\infty
      \prod_{i=1}^\infty
        \biggl[
          1 - \alpha
            + \frac{\alpha}{i!}
              \int_0^\infty
                \frac{\ee^{-y}\,y^i}
                     {\bigl(1 + (u/y)^\beta\bigr)^m}\,
              \dd y
        \biggr]\,
    \dd u,
\end{align*}
where we substitute $v=m^{-1/\beta}\,t$ and apply the Beta
function~$B(x,y) = \Gamma(x)\,\Gamma(y)/\Gamma(x+y)$ in the second
equality.
\end{remark}

It is known that $\alpha$-Ginibre point processes converge weakly to
the homogeneous Poisson point process with the same intensity as
$\alpha\to0$ (see~\cite{Gold10}).
The following is an extension of Proposition~5 in~\cite{MiyoShir14b},
where the case of $F_H=\mathrm{Exp}(1)$ is considered.

\begin{proposition}
Let $C^{(\textnormal{$\alpha$-GPP})}(\beta, F_H)$ denote the asymptotic
constant on the right-hand side of \eqref{eq:GiniAsym}.
Then, for any propagation effect distribution~$F_H$ satisfying
Condition~(A) of Theorem~\ref{thm:general},
\begin{equation}\label{eq:alpha-constant}
  \lim_{\alpha\downarrow0}C^{(\textnormal{$\alpha$-GPP})}(\beta, F_H)
  = \frac{\beta}{\pi}\,\sin\frac{\pi}{\beta}.
\end{equation}
\end{proposition}

Note that the right-hand side of \eqref{eq:alpha-constant} is just the
asymptotic constant in Corollary~\ref{cor:Poisson} for the homogeneous
Poisson-based model.

\begin{proof}
The proof essentially follows the similar line to that of
Proposition~5 in \cite{MiyoShir14b}.
Since the asymptotic constant~\eqref{eq:GH15Thm4} in
Theorem~\ref{thm:general} does not depend on the intensity of the
point process, we here choose $\lambda = \alpha/\pi$.
Then, $Y_i\sim\Gam(i+1,\alpha^{-1})$ in
Proposition~\ref{prp:Kostlan}~(ii) is replaced with
$Y_i\sim\Gam(i+1,1)$.
Clearly, the right-hand side of \eqref{eq:GiniAsym} is equal to
\begin{equation}\label{eq:prp4prf1}
  C^{(\textnormal{$\alpha$-GPP})}(\beta, F_H)
  = \frac{\alpha\,\Exp({H_1}^{1/\beta})}{\Gamma(1+1/\beta)}
    \int_0^\infty
      \prod_{i=1}^\infty
        \biggl\{
          1 - \alpha\,
              \Exp\biggl[
                1 - \Lpl_H\Bigl(\Bigl(\frac{t}{Y_i}\Bigr)^\beta\Bigr)
              \biggr]
        \biggr\}\,
    \dd t.
\end{equation}
We here use the fact that, for any $\delta>0$, there exists an
$x_\delta\in(0,1)$ such that $\ee^{-(1+\delta) x} \leqsl 1-x \leqsl
\ee^{-x}$ for $x\in[0,x_\delta]$.
Thus, for $\alpha\in(0, x_\delta]$, the integrand above has upper and
lower bounds such as
\begin{align}\label{eq:prp4prf2}
  &\exp\biggl\{
     -(1+\delta)\,\alpha
      \sum_{i=1}^\infty
        \Exp\biggl[
          1 - \Lpl_H\Bigl(\Bigl(\frac{t}{Y_i}\Bigr)^\beta\Bigr)
        \biggr]
   \biggr\}
  \\ 
  &\le \prod_{i=1}^\infty
         \biggl\{
           1 - \alpha\,
               \Exp\biggl[
                 1 - \Lpl_H\Bigl(\Bigl(\frac{t}{Y_i}\Bigr)^\beta\Bigr)
               \biggr]
         \biggr\}
   \le \exp\biggl\{
         -\alpha
          \sum_{i=1}^\infty
            \Exp\biggl[
              1 - \Lpl_H\Bigl(\Bigl(\frac{t}{Y_i}\Bigr)^\beta\Bigr)
            \biggr]
       \biggr\}.
  \nonumber
\end{align}
Here, applying the density function of $Y_i\sim\Gam(i+1,1)$, $i\in\N$,
we have
\begin{align}\label{eq:alphato0}
  \sum_{i=1}^\infty
    \Exp\biggl[
      1 - \Lpl_H\Bigl(\Bigl(\frac{t}{Y_i}\Bigr)^\beta\Bigr)
    \biggr]
  &= \int_0^\infty  
       (1 - \ee^{-y})\,
       \biggl[
         1 - \Lpl_H\Bigl(\Bigl(\frac{t}{y}\Bigr)^\beta\Bigr)
       \biggr]\,
     \dd y
  \\
  &= t \int_0^\infty
         (1 - \ee^{-t u})\,
         \bigl[
           1 - \Lpl_H(u^{-\beta})
         \bigr]\,
       \dd u,
  \nonumber
\end{align}
where the last equality follows by substituting $u = y/t$.
From the last expression above, we have for any $t>0$,
\begin{equation}\label{eq:prp4prf3}
  t \int_0^\infty
      \bigl[ 1 - \Lpl_H(u^{-\beta}) \bigr]\,
    \dd u
  - 1      
  \le \eqref{eq:alphato0}
  \le t \int_0^\infty
          \bigl[ 1 - \Lpl_H(u^{-\beta}) \bigr]\,
        \dd u,
\end{equation}
and the common integral on both the sides reduces to
\begin{align}\label{eq:prp4prf4}
  \int_0^\infty
    \bigl[ 1 - \Lpl_H(u^{-\beta}) \bigr]\,
  \dd u
  &= \Exp\biggl[
       \int_0^\infty
         \bigl( 1 - \ee^{-u^{-\beta} H_1} \bigr)\,
       \dd u
     \biggr]
  \\
  &= \frac{\Exp({H_1}^{1/\beta})}{\beta}
     \int_0^\infty (1-\ee^{-v})\,v^{-1-1/\beta}\,\dd v
  = \Exp({H_1}^{1/\beta})\,
    \Gamma\Bigl(1 - \frac{1}{\beta}\Bigr),
  \nonumber
\end{align}
where the second equality follows by substituting $v =
H_1\,u^{-\beta}$ and the last equality follows from the integration by
parts.
Hence, applying \eqref{eq:prp4prf2}--\eqref{eq:prp4prf4} to
\eqref{eq:prp4prf1} and using $\Gamma(x)\,\Gamma(1-x) =
\pi\,\csc\pi\,x$ for $x\in(0,1)$, we obtain
\[
  \frac{1}{1 + \delta}\,
  \frac{\beta}{\pi}\,\sin\frac{\pi}{\beta}
  \le C^{(\textnormal{$\alpha$-GPP})}(\beta, F_H)
  \le \ee^{\alpha}\,\frac{\beta}{\pi}\,\sin\frac{\pi}{\beta},
\]
The assertion follows as $\alpha\downarrow0$ since $\delta$ is
arbitrary.
\end{proof}

\subsection{A counterexample}

Finally in this section, we give a simple counterexample that violates
Condition~(B) of Theorem~\ref{thm:general}.
Let $T$ denote a random variable with density function
$f_T(t)=(a-1)\,t^{-a}$, $t\geqsl1$, for $a\in(1,2)$.
Note that $\Exp T = \infty$.
Given a sample of $T$, we consider the mixed and randomly shifted
lattice $\Phi = (\Z\times T\,\Z) + U_T$, where $U_T$ denotes a
uniformly distributed random variable on $[0,1]\times[0,T]$.
The intensity $\lambda$ of $\Phi$ is then $\lambda=\Exp(1/T) =
(a-1)/a<\infty$.
For any nonnegative and measurable function~$g$, the definition of the
Palm probability gives
\begin{align*}
  \Exp^o g(T)
  &= \frac{1}{\lambda}\,
     \Exp\bigl(g(T)\,\Phi(I)\bigr)
   = \frac{1}{\lambda}\,
     \Exp\bigl(g(T)\,\Exp(\Phi(I)\mid T)\bigr)
   = \frac{1}{\lambda}\,\Exp\biggl(\frac{g(T)}{T}\biggr),
\end{align*}
where $I=[0,1]^2$.
Hence, applying $R(o)^2 = (1+T^2)/4$ to the above,
we have
\begin{align*}
  \Exp^o\bigl(R(o)^2\bigr)
  = \frac{1}{4\,\lambda}\,
    \Exp\biggl(
      \frac{1}{T}+T
    \biggr)
  = \frac{1}{4\,\lambda}\,(\lambda + \Exp(T)) = \infty.
\end{align*}

\section{Tail asymptotics for bounded path-loss
  models}\label{sec:bounded}

In this section, we consider bounded and regularly varying path-loss
functions.
We assume that the distribution of propagation effects is light-tailed
and restrict ourselves to two cases of the point process~$\Phi$; one
is a homogeneous Poisson point process on $\R^d$ and the other is an
$\alpha$-Ginibre point process on $\C\simeq\R^2$.
In both the cases, we derive the same logarithmically asymptotic upper
bound on the SIR tail distributions.
Furthermore, when $\Phi$ is a homogeneous Poisson point process and
the propagation effects are exponentially distributed, a
logarithmically asymptotic lower bound with the same order as the
upper bound is obtained.
We first impose an assumption on the path-loss function~$\ell$.

\begin{assumption}\label{asm:bounded}
$\ell$ is nonincreasing, bounded on $[0,\infty)$ and regularly varying
at infinity with index~$-d\beta$, $\beta>1$, in the sense that (see,
e.g., \cite{BingGoldTeug87,Sene76})
\[
  \lim_{x\to\infty}\frac{\ell(t\,x)}{\ell(x)} = t^{-d\,\beta}
  \quad\text{for all $t>0$.}
\]
\end{assumption}

In what follows, we suppose for simplicity that $\ell$ is bounded by
$1$; that is, $\ell(r)\leqsl 1$ for $r\in[0,\infty)$.
Let $g(s) = 1/\ell(s^{1/d})$, $s\geqsl 0$.
By Assumption~\ref{asm:bounded} above, we see that the function~$g$ is
nondecreasing and regularly varying at infinity with index~$\beta$.
Thus, we can define an asymptotic inverse function~$h$ of $g$
satisfying $g(h(z)) \sim h(g(z)) \sim z$ as $z\to\infty$ (see, e.g.,
\cite[Sec.~1.5]{BingGoldTeug87}, \cite[Chap.~1]{Sene76}).
The function~$h$ is asymptotically unique and also regularly varying
at infinity with index~$1/\beta$.
For example, when $\ell(r) = (1+r^{d\beta})^{-1}$, then $g(s) =
1/\ell(s^{1/d}) = 1+s^\beta$ and we can take $h(z) = z^{1/\beta}$.
More generally, if $\ell(r) = \bigl(1 +
r^{d\beta}\,[\log(1+r)]^a\bigr)^{-1}$ with $a \geqsl -d\,\beta$, then
$g(s) = 1 + s^\beta\,\bigl[\log(1+s^{1/d})\bigr]^a$ and we can take
$h(z) = z^{1/\beta}\,(d\beta/\log z)^{a/\beta}$.
The following theorem states that the SIR tail probability
$\Prb(\SIR_o>\theta)$ is asymptotically bounded above by
$\ee^{-\Theta(h(\theta))}$ as $\theta\to\infty$.

\begin{theorem}\label{thm:bound_pathloss}
For the cellular network model described in Section~\ref{sec:intro}
with the path-loss function satisfying Assumption~\ref{asm:bounded},
we suppose that the distribution~$F_H$ of the propagation effects $H_i$,
$i\in\N$, satisfies the following:
\begin{enumerate}[(a)]
\item\label{cond:a}
  It is light-tailed; that is, there exists a (possibly infinite)
  $\zeta_0>0$ such that the moment generating function $\Mgf_H(\zeta)
  = \Exp\ee^{\zeta H_1}$ is finite for $\zeta < \zeta_0$.
\item\label{cond:b}
  The Laplace transform~$\Lpl_H$ satisfies
  $\log\Lpl_H(s) = o(L_{1/\beta}(s))$ as $s\to\infty$ for any
  regularly varying function $L_{1/\beta}$ with index~$1/\beta$.
\end{enumerate}  

If $\Phi=\{X_i\}_{i\in\N}$ is a homogeneous Poisson point process on
$\R^d$ with positive and finite intensity, then using the function~$h$
defined above, we have
\begin{equation}\label{eq:bound_Poisson}
  \limsup_{\theta\to\infty}
    \frac{1}{h(\theta)}\,\log\Prb(\SIR_o>\theta)
  \leqsl - \Gamma\biggl(1-\frac{1}{\beta}\biggr)\,
        {\zeta_0}^{1/\beta}\,\Exp({H_1}^{1/\beta}),
\end{equation}
where $\zeta_0$ is the critical value for the existence of the moment
generating function~$\Mgf_H$ of $H_i$, $i\in\N$.
Moreover, if $d=2$ and $\Phi=\{X_i\}_{i\in\N}$ is an $\alpha$-Ginibre
point process on $\C\simeq\R^2$, we have \eqref{eq:bound_Poisson} as
well.
\end{theorem}

Note that, if $\zeta_0=\infty$, the SIR tail probability
$\Prb(\SIR_o>\theta)$ decays faster than $\ee^{-\Theta(h(\theta))}$ as
$\theta\to\infty$.

\begin{proof}
Since $F_H$ is light-tailed, Markov's inequality yields
$\Bar{F_H}(x) = \Prb(\ee^{\zeta H_i} > \ee^{\zeta x}) \leqsl
\Mgf_H(\zeta)\,\ee^{-\zeta x}$ for $\zeta \in (0,\zeta_0)$.
Thus, by $\ell(r)\leqsl 1$, \eqref{eq:BarF_H} with replacing
$|X_i|^{-d\beta}$ with $\ell(|X_i|)$ is bounded above as
\begin{align}\label{eq:bound_pathloss}
  \Prb(\SIR_o > \theta)
  &\leqsl \Mgf_H(\zeta)\,
       \Exp\exp\biggl\{
         - \frac{\zeta\,\theta}{\ell(|X_1|)}
           \sum_{i=2}^\infty H_i\,\ell(|X_i|)
       \biggr\}
  \\
  &= \Mgf_H(\zeta)\,
     \Exp\biggl[
        \prod_{i=2}^\infty
          \Lpl_H\biggl(
            \zeta\,\theta\,\frac{\ell(|X_i|)}{\ell(|X_1|)}
          \biggr)
      \biggr]
  \leqsl \frac{\Mgf_H(\zeta)}{\Lpl_H(\zeta\,\theta)}\,
      \Exp\biggl[
        \prod_{i=1}^\infty
          \Lpl_H(\zeta\,\theta\,\ell(|X_i|))
      \biggr].
  \nonumber
\end{align}

Suppose that $\Phi=\{X_i\}_{i\in\N}$ is a homogeneous Poisson point
process.
Since $\SIR_o$ is invariant to the intensity~$\lambda$, we choose
$\lambda = {\pi_d}^{-1}$.
Then, applying the probability generating functional (see, e.g.,
\cite[Sec.~9.4]{DaleVere08}) to the expectation above, we have
\begin{align}\label{eq:bound_Poisson1}
  \Exp\biggl[
    \prod_{i=1}^\infty
      \Lpl_H(\zeta\,\theta\,\ell(|X_i|))
  \biggr]
  &= \exp\biggl\{
       - \frac{1}{\pi_d}
         \int_{\R^d}
           \bigl[
             1 - \Lpl_H(\zeta\,\theta\,\ell(|x|))
           \bigr]\,
         \dd x
     \biggr\}
  \\
  &= \exp\biggl\{
       - d \int_0^\infty
             \bigl[
               1 - \Lpl_H(\zeta\,\theta\,\ell(r))
             \bigr]\,
             r^{d-1}\,
           \dd r
     \biggr\}
  \nonumber\\
  &= \exp\biggl\{
       - \int_0^\infty
           \bigl[
             1 - \Lpl_H(\zeta\,\theta\,\ell(s^{1/d}))
           \bigr]\,
         \dd s
     \biggr\},
  \nonumber
\end{align}
where the last equality follows by the substitution of $s=r^d$.
Here, we set $\theta = g(z) = 1/\ell(z^{1/d})$, $z\geqsl 0$.
Note that $\theta\to\infty$ as $z\to\infty$.
Then, since $h(g(z)) \sim z$ as $z\to\infty$, we have
\begin{align}\label{eq:bound_Poisson2}
  \lim_{\theta\to\infty}
    \frac{1}{h(\theta)}
    \int_0^\infty
      \bigl[
        1 - \Lpl_H\bigl(\zeta\,\theta\,\ell(s^{1/d})\bigr)
      \bigr]\,
    \dd s
  &= \lim_{z\to\infty}
       \frac{1}{z}
       \int_0^\infty
         \biggl[
           1 - \Lpl_H\biggl(
                 \zeta\,\frac{g(z)}{g(s)}
               \biggr)
         \biggr]\,
       \dd s
  \\
  &= \lim_{z\to\infty}
       \int_0^\infty
         \biggl[
           1 - \Lpl_H\biggl(
                 \zeta\,\frac{g(z)}{g(z\,t)}
               \biggr)
         \biggr]\,
       \dd t,
  \nonumber
\end{align}
where $t= s/z$ is substituted in the second equality.
We will confirm later whether the dominated convergence theorem is
applicable in the last expression above and we now admit it.
The regular variation of $g$ with index~$\beta$ then yields
\begin{align}\label{eq:Thm2Prf4}
  \eqref{eq:bound_Poisson2}
  &= \int_0^\infty
       \bigl[1 - \Lpl_H(\zeta\,t^{-\beta})\bigr]\,
     \dd t
  = \zeta^{1/\beta}\,\Exp({H_1}^{1/\beta})\,
     \Gamma\biggl(1-\frac{1}{\beta}\biggr),
\end{align}
where the second equality follows by the similar procedure to
\eqref{eq:prp4prf4}.
Hence, applying \eqref{eq:bound_Poisson1}--\eqref{eq:Thm2Prf4} to
\eqref{eq:bound_pathloss} and taking $\zeta\to\zeta_0$, we obtain
\eqref{eq:bound_Poisson} since $\log\Lpl_H(\zeta\,\theta)/h(\theta)
\to 0$ as $\theta\to\infty$ by Condition~(b) of the theorem.

Let us now show that the dominated convergence theorem is applicable
in \eqref{eq:Thm2Prf4}.
Since $g$ is regularly varying with index~$\beta$, we have $g(z) =
z^\beta\,L_0(z)$ with a slowly varying function~$L_0$, for which we
can take a constant~$B>0$ such that
\[
  L_0(z)
  = \exp\biggl(
      \eta(z) + \int_B^z \frac{\epsilon(u)}{u}\,\dd u
    \biggr),
  \quad z\geqsl B,
\]
where $\eta(z)$ is bounded and converges to a constant as
$z\to\infty$, and $\epsilon(u)$ is bounded and converges to zero
as $u\to\infty$ (see, e.g., \cite[Sec.~1.3]{BingGoldTeug87} or
\cite[Chap.~1]{Sene76}).
We define constants $\eta^*$ and $\epsilon^*$ as
\[
  \eta^* = \sup_{z\geqsl B}|\eta(z)|,
  \qquad
  \epsilon^* = \sup_{z\geqsl B}|\epsilon(z)|.
\]
Note here that we can take $B$ large enough such that $\epsilon^* <
\beta - 1$.
Then, for $z\geqsl B$ and $t\geqsl 1$, we have
\begin{equation}\label{eq:representation}
  \frac{g(z)}{g(z\,t)}
  \leqsl t^{-\beta}\,\ee^{2\eta^*}\,
         \exp\biggl(\epsilon^*\int_z^{zt}\frac{\dd u}{u}\biggr)
  = \ee^{2\eta^*}\,t^{-(\beta-\epsilon^*)},
\end{equation}
so that the integrand of the last expression in
\eqref{eq:bound_Poisson2} satisfies
\[
  1 - \Lpl_H\biggl(
        \zeta\,\frac{g(z)}{g(z\,t)}
      \biggr)
  \leqsl \ind{(0,1]}(t)
         + \bigl[
             1 - \Lpl_H(
                   \zeta\,\ee^{2\eta^*}\,t^{-(\beta-\epsilon^*)}
                 )
           \bigr]\,
           \ind{(1,\infty)}(t).
\]        
Similar to \eqref{eq:prp4prf4} (and \eqref{eq:Thm2Prf4}), the integral
of the second term on the right-hand side above amounts to
\begin{align*}
  \int_1^\infty
    \bigl[
        1 - \Lpl_H(\zeta\,\ee^{2\eta^*}\,t^{-(\beta - \epsilon^*)})
    \bigr]\,
    \dd t
  &\leqsl \Exp\biggl[
            \int_0^\infty
         \bigl(
           1 - \ee^{-\zeta \ee^{2\eta^*} H_1\,t^{-(\beta - \epsilon^*)}}
         \bigr)\,
       \dd t
     \biggr]
  \\
  &= (\zeta\,\ee^{2\eta^*})^{1/(\beta-\epsilon^*)}\,
     \Exp({H_1}^{1/(\beta-\epsilon^*)})\,
     \Gamma\biggl(1-\frac{1}{\beta-\epsilon^*}\biggr)
  < \infty,
\end{align*}
and the dominated convergence theorem is applicable.

Next, we show \eqref{eq:bound_Poisson} when $d=2$ and
$\Phi=\{X_i\}_{i\in\N}$ is an $\alpha$-Ginibre point process.
Recall Proposition~\ref{prp:Kostlan}~(i), which states that
$\{|X_i|^2\}_{i\in\N}$ has the same distribution as
$\{\check{Y}_i\}_{i\in\N}$ and each $\check{Y}_i$ is extracted from
$\{Y_i\}_{i\in\N}$ with probability~$\alpha$ independently, where
$Y_i\sim\Gam(i, \alpha^{-1})$, $i\in\N$, are mutually independent.
Applying this to \eqref{eq:bound_pathloss}, we have
\[
  \Prb(\SIR_o > \theta)
  \leqsl \frac{\Mgf_H(\zeta)}{\Lpl_H(\zeta\,\theta)}
      \prod_{i=1}^\infty
        \Exp\bigl[
          1 - \alpha
          + \alpha\,\Lpl_H\bigl(\zeta\,\theta\,\ell({Y_i}^{1/2})\bigr)
        \bigr],
\]
so that, using $\log x\leqsl x-1$,
\[
  \log\Prb(\SIR_o > \theta)
  \leqsl \log\frac{\Mgf_H(\zeta)}{\Lpl_H(\zeta\,\theta)}
      - \alpha\sum_{i=1}^\infty
          \Exp\bigl[
            1 - \Lpl_H\bigl(\zeta\,\theta\,\ell({Y_i}^{1/2})\bigr)
          \bigr].
\]
Hence, applying the density function of
$Y_i\sim\Gam(i,\alpha^{-1})$, $i\in\N$,
\begin{align*}
  \alpha\sum_{i=1}^\infty
    \Exp\bigl[
      1 - \Lpl_H\bigl(\zeta\,\theta\,\ell({Y_i}^{1/2})\bigr)
    \bigr]
  &= \alpha\sum_{i=1}^\infty
     \int_0^\infty
       \frac{(y/\alpha)^{i-1}\,\ee^{-y/\alpha}}{(i-1)!}\,
       \bigl[
         1 - \Lpl_H(\zeta\,\theta\,\ell(y^{1/2}))
       \bigr]\,
     \dd y
  \\
  &= \int_0^\infty
       \bigl[
         1 - \Lpl_H(\zeta\,\theta\,\ell(y^{1/2}))
       \bigr]\,
     \dd y,
\end{align*}
which is the same expression as the exponent of
\eqref{eq:bound_Poisson1} and leads to the same result.
\end{proof}  

\begin{remark}
We can see that many practical distributions satisfy Condition~(b) of
Theorem~\ref{thm:bound_pathloss}.
Since $\Lpl_H(s) \geqsl \Exp\bigl(\ee^{-sH_1}\,\ind{\{H_1\leqsl 1/s\}}\bigr)
\geqsl \ee^{-1}\,F_H(1/s)$, we have $|\log\Lpl_H(s)| \leqsl 1-\log
F_H(1/s)$.
Thus, for example, if $F_H(x) \geqsl c\,x^a$ for $x\in[0,\epsilon]$ with
$c>0$, $a\geqsl0$ and $\epsilon>0$, then $|\log\Lpl_H(s)| = O(\log s)$ as
$s\to\infty$.
On the other hand, a counterexample is such that there exists a
constant~$\epsilon>0$ with $F_H(\epsilon)=0$.
Then, $\Lpl_H(s)\leqsl \ee^{-\epsilon s}$ and we have
$|\log\Lpl_H(s)|\geqsl\epsilon\,s$.
\end{remark}

\begin{remark}
When $F_H=\mathrm{Exp}(1)$, we have $\Mgf_H(s) = (1-s)^{-1}$, $s<1$.
In this case, $\zeta_0=1$ and $\Exp({H_1}^{1/\beta}) =
\Gamma(1+1/\beta)$ in Theorem~\ref{thm:bound_pathloss}, so that, by
$\Gamma(1+1/\beta)\,\Gamma(1-1/\beta) = (\pi/\beta)\,\csc(\pi/\beta)$,
Theorem~\ref{thm:bound_pathloss} reduces to
\begin{equation}\label{eq:bound_Exp}
  \limsup_{\theta\to\infty}
    \frac{1}{h(\theta)}\,\log\Prb(\SIR_o>\theta)
    \leqsl - \frac{\pi}{\beta}\,\csc\frac{\pi}{\beta}.
\end{equation}
\end{remark}

When $F_H=\mathrm{Exp}(1)$ and $\Phi$ is a homogeneous Poisson point
process with positive and finite intensity, we can show that the tail
distribution of the SIR has a logarithmically asymptotic lower bound
which has the same order as the upper bound~\eqref{eq:bound_Exp}.

\begin{proposition}
For the cellular network model with the path-loss function satisfying
Assumption~\ref{asm:bounded}, when $F_H = \mathrm{Exp}(1)$ and $\Phi =
\{X_i\}_{i\in\N}$ is a homogeneous Poisson point process on $\R^d$
with positive and finite intensity, we have
\begin{equation}\label{eq:lowerbound_Poisson}
  \liminf_{\theta\to\infty}
    \frac{1}{h(\theta)}\,\log\Prb(\SIR_o > \theta)
  \geqsl - \Exp[\ell(|X_1|)^{-1/\beta}]\,
        \frac{\pi}{\beta}\,\csc\frac{\pi}{\beta}.
\end{equation}
\end{proposition}

\begin{proof}
Applying $\Bar{F_H}(x) = \ee^{-x}$, $x\geqsl0$, and $\Lpl_H(s) =
(1+s)^{-1}$, $s\geqsl0$, we rewrite \eqref{eq:BarF_H} as
\[
  \Prb(\SIR_o > \theta)
  = \Exp\biggl[
      \prod_{i=2}^\infty
        \biggl(
          1 + \theta\,\frac{\ell(|X_i|)}{\ell(|X_1|)}
        \biggr)^{-1}
    \biggr].
\]
Let the intensity of the Poisson point process be
$\lambda={\pi_d}^{-1}$.
By concavity of logarithmic functions, Jensen's inequality yields
\begin{align*}
  \log\Prb(\SIR_o > \theta)
  &\geqsl \Exp\biggl[
         \log\Exp\biggl[
           \prod_{i=2}^\infty
             \biggl(
               1 + \theta\,\frac{\ell(|X_i|)}{\ell(|X_1|)}
             \biggr)^{-1}
         \biggm| |X_1|\biggr]
     \biggr]
  \\
  &= -\frac{1}{\pi_d}\,
      \Exp\biggl[
        \int_{|x|>|X_1|}
          \biggl\{
            1 - \biggl(1 + \theta\,\frac{\ell(|x|)}{\ell(|X_1|)}\biggr)^{-1}
          \biggr\}\,
        \dd x
      \biggr]
  \\
  &= -\Exp\biggl[
        \int_{|X_1|^d}^\infty
          \biggl\{
            1 - \biggl(1 + \theta\,\frac{\ell(s^{1/d})}{\ell(|X_1|)}\biggr)^{-1}
          \biggr\}\,
        \dd s
      \biggr],
\end{align*}
where we apply the probability generating functional to the
conditional expectation given $|X_1|$ in the first equality and use
the similar procedure to \eqref{eq:bound_Poisson1} in the last
equality.
Here, similar to the proof of Theorem~\ref{thm:bound_pathloss}, we set
$\theta = g(z)$.
Then, since $h(g(z)) \sim z$ as $z\to\infty$, we have
\begin{align}\label{eq:lower_prf1}
  \liminf_{\theta\to\infty}
    \frac{1}{h(\theta)}\,
    \log\Prb(\SIR_o > \theta)
  &\geqsl -\limsup_{z\to\infty}
         \frac{1}{z}\,
         \Exp\biggl[
           \int_{|X_1|^d}^\infty
             \biggl\{
               1 - \biggl(
                     1 + \frac{1}{\ell(|X_1|)}\,
                         \frac{g(z)}{g(s)}
                   \biggr)^{-1}
             \biggr\}\,
           \dd s
         \biggr]
  \\
  &= -\limsup_{z\to\infty}
        \Exp\biggl[
           \int_{|X_1|^d/z}^\infty
             \biggl\{
               1 - \biggl(
                     1 + \frac{1}{\ell(|X_1|)}\,
                         \frac{g(z)}{g(z\,t)}
                   \biggr)^{-1}
             \biggr\}\,
           \dd t
        \biggr],
  \nonumber
\end{align}
where $t = s/z$ is substituted in the last equality.
We will confirm later that the dominated convergence theorem is
applicable to the above and we have
\begin{equation}\label{eq:lower_prf2}
  \lim_{z\to\infty}
     \Exp\biggl[
       \int_{|X_1|^d/z}^\infty
         \biggl\{
           1 - \biggl(
                 1 + \frac{1}{\ell(|X_1|)}\,
                     \frac{g(z)}{g(z\,t)}
               \biggr)^{-1}
         \biggr\}\,
       \dd t
     \biggr]
  = \Exp\biggl[
       \int_0^\infty
         \frac{\dd t}{1+\ell(|X_1|)\,t^\beta}
     \biggr].
\end{equation}
Furthermore, substituting $u = \ell(|X_1|)\,t^\beta$,
\begin{equation}\label{eq:lower_prf3}
  \int_0^\infty
    \frac{\dd t}{1+\ell(|X_1|)\,t^\beta}
  = \frac{1}{\beta\,\ell(|X_1|)^{1/\beta}}
    \int_0^\infty
      \frac{u^{1/\beta-1}}{1+u}\,
    \dd u
  = \frac{1}{\ell(|X_1|)^{1/\beta}}\,
    \frac{\pi}{\beta}\,\csc\frac{\pi}{\beta},
\end{equation}
which, together with \eqref{eq:lower_prf1} and \eqref{eq:lower_prf2},
leads to \eqref{eq:lowerbound_Poisson}.

It remains to show whether the dominated convergence theorem is
applicable in \eqref{eq:lower_prf2}.
Applying inequality~\eqref{eq:representation}, we have
\begin{align*}
  \biggl\{
    1 - \biggl(
          1 + \frac{1}{\ell(|X_1|)}\,
              \frac{g(z)}{g(z\,t)}
        \biggr)^{-1}
  \biggr\}\,\ind{\{|X_1|\leqsl (zt)^{1/d}\}}
  &\leqsl \ind{(0,1]}(t)
          + \Bigl(
              1 + \ee^{-2\eta^*}\,\ell(|X_1|)\,t^{\beta-\epsilon^*}
            \Bigr)^{-1}\,
            \ind{(1,\infty)}(t).
\end{align*}
Similar to \eqref{eq:lower_prf3}, the integral of the second term on
the right-hand side above amounts to
\begin{align*}
  \int_1^\infty
    \frac{\dd t}
         {1 + \ee^{-2\eta^*}\,\ell(|X_1|)\,t^{\beta-\epsilon^*}}
  &\leqsl \frac{\ee^{2\eta^*/(\beta-\epsilon^*)}}
               {(\beta - \epsilon^*)\,\ell(|X_1|)^{1/(\beta-\epsilon^*)}}
          \int_0^\infty
            \frac{u^{-1+1/(\beta-\epsilon^*)}}{1+u}\,
          \dd u
  \\
  &= \frac{\ee^{2\eta^*/(\beta-\epsilon^*)}}
          {\ell(|X_1|)^{1/(\beta-\epsilon^*)}}\,
     \frac{\pi}{\beta-\epsilon^*}\,\csc\frac{\pi}{\beta-\epsilon*}.
\end{align*}
It then suffices to show that
$\Exp[\ell(|X_1|)^{-1/(\beta-\epsilon^*)}] < \infty$.
The regular variation of $\ell$ with index $-d\beta$ implies that
$\ell(r) = r^{-d\beta}\,\Tilde{L}_0(r)$ with another slowly varying
function~$\Tilde{L}_0$, for which we can take a constant~$\Tilde{B}>1$
such that
\[
  \Tilde{L}_0(r)
  = \exp\biggl(
      \Tilde{\eta}(r)
      + \int_{\Tilde{B}}^r \frac{\Tilde{\epsilon}(t)}{t}\,\dd t
    \biggr),
 \quad r\geqsl \Tilde{B},
\]
where $\Tilde{\eta}(r)$ is bounded and converges to a constant as
$r\to\infty$, and $\Tilde{\epsilon}(t)$ is bounded and converges
to zero as $t\to\infty$.
Define constants~$\Tilde{\eta}_\beta$ and $\Tilde{\epsilon}_\beta$ as
\[
  \Tilde{\eta}_\beta
  = \sup_{x\geqsl\Tilde{B}}\frac{|\Tilde{\eta}(x)|}{\beta-\epsilon^*},
  \qquad
  \Tilde{\epsilon}_\beta
  = \sup_{x\geqsl\Tilde{B}}\frac{|\Tilde{\epsilon}(x)|}{\beta-\epsilon^*}.
\]
Since $\ell$ is nonincreasing, we have
\begin{align}\label{eq:last}
  \ell(|X_1|)^{-1/(\beta-\epsilon^*)}
  &\leqsl \ell(\Tilde{B})^{-1/(\beta-\epsilon^*)}\,
          \ind{[0,\Tilde{B}]}(|X_1|)
          + |X_1|^d\,\Tilde{L}_0(|X_1|)^{-1/(\beta-\epsilon^*)}\,
            \ind{(\Tilde{B},\infty)}(|X_1|)
  \\
  &\leqsl \ell(\Tilde{B})^{-1/(\beta-\epsilon^*)}
          + \ee^{\Tilde{\eta}_\beta}\,|X_1|^{d+\Tilde{\epsilon}_\beta},
  \nonumber
\end{align}
which completes the proof since $|X_1|$ has any order of moments.
\end{proof}

\begin{remark}
During the preparation of the first draft, the authors have found that
the result in \cite[Sec.~IV-B]{GuoHaenGant16} corresponds to our
\eqref{eq:bound_Exp} and \eqref{eq:lowerbound_Poisson} for the
homogeneous Poisson-based model with $d=2$, $F_H=\mathrm{Exp}(1)$ and
$\ell(r) = (1 + r^{2\beta})^{-1}$.
We here deal with a much wider class of path-loss functions than that
of power-law decaying functions.
\end{remark}

\begin{remark}
The results in this section hold when we relax the nonincreasing
property of $\ell$ in Assumption~\ref{asm:bounded} such that, for any
finite $\Tilde{B}>0$, there exists an $\varepsilon_{\Tilde{B}}>0$ such
that $\ell(r)\geqsl\varepsilon_{\Tilde{B}}$ for $r\in[0,\Tilde{B}]$,
as we remain to assume the boundedness and regular variation.
In this case, the proofs remain the same except for replacing
$\ell(\Tilde{B})$ in~\eqref{eq:last} with $\varepsilon_{\Tilde{B}}$.
\end{remark}


\end{document}